\documentclass[12pt]{amsart}
\usepackage{amsmath, amssymb, amsthm, bm, float, mathtools, enumitem, caption}
\usepackage{hyperref}

\usepackage[noabbrev,capitalize]{cleveref}
\usepackage{adjustbox}
\crefname{equation}{}{}
\usepackage{ytableau}
\usepackage{graphicx, tikz}
\usepackage[textheight=8.5in, textwidth=6.75in]{geometry}
\usepackage{subcaption}
\usepackage{microtype}
\usepackage{stmaryrd}
\usepackage{mathtools}

\newtheorem{theorem}{Theorem}[section]
\newtheorem{lemma}[theorem]{Lemma}
\newtheorem{corollary}[theorem]{Corollary}

\newtheorem*{conjecture*}{Conjecture}

\theoremstyle{definition}

\theoremstyle{remark}

\newtheorem*{example}{Example}

\numberwithin{equation}{section}

\newcommand{\N}{\mathbb N}

\makeatletter
\DeclarePairedDelimiterX{\pmodx}[1]{(}{)}{{\operator@font mod}\mkern6mu#1}
\renewcommand{\pmod}{%
  \allowbreak
  \if@display\mkern18mu\else\mkern8mu\fi
  \pmodx
}
\makeatother

\title[Rademacher-type formula and higher order Tur\'{a}n inequalities]{ Rademacher-type formula and higher order Tur\'{a}n inequalities for $\ell$-regular overpartitions}
\date{\today}
\author{Priyanka Dey, Ajit Singh, \and Gurinder Singh}
\address{Dept. of Mathematics \& Computing, Indian Institute of Technology (Indian School of Mines) Dhanbad, Jharkhand, India.}
\email{25dr0355@iitism.ac.in}
\email{ajit94@iitism.ac.in}
\address{Postdoctoral Research Station of Mathematics, Hebei Normal University, Shijiazhuang 050024, P.R. China.}
\email{gurindermaan1018@gmail.com}

\begin{document}
\begin{abstract} 
For $\ell\geq 2$, let $\overline{A}_\ell(n)$ count the number of overpartitions of $n$ with no parts divisible by $\ell$. In this article,  we employ the circle method to derive a Rademacher-type formula for $\overline{A}_\ell(n)$, when $\ell$ is a squarefree odd integer. As an application, we derive higher order Tu\'{r}an inequalities for the $\ell$-regular overpartition function using a result of Griffin, Ono, Rolen, and Zagier.
\end{abstract}
\maketitle
\section{Introduction and Statement of Results}

Integer partitions play many roles in mathematics. They are fundamental objects in combinatorics, geometry, mathematical physics, number theory,  and representation theory. For example, they arise in the study of class numbers of imaginary quadratic fields \cite{OS1997}, the Seiberg-Witten theory of random partitions developed by Nekrasov and Okounkov \cite{NO2006}, Ramanujan's partition congruences \cite{GKS1990}, and the representation theory of the symmetric group \cite{JK1981}, to name a few applications. An important question in the theory of partitions is to determine exact formulas or asymptotics for functions such as $p(n)$. In this paper, we prove the Rademacher-type formula for $\ell$-regular overpartitions.

To make this precise, we recall the fundamental definitions (for example, see \cite{And1998}).
A non-increasing sequence of natural numbers $\lambda=(\lambda_1,\lambda_2,\dots,\lambda_{r})$ is called a \emph{partition} of $n$, where $\lambda_1$ through $\lambda_{r}$ are called its \emph{parts}, provided these add up to $n$. Denote the number of partitions of $n$ by $p(n)$. The corresponding generating function is
$$\sum_{n=0}^\infty p(n)q^n=\prod_{n=1}^\infty \frac{1}{(1-q^n)}=\frac{1}{(q;q)_\infty},$$
where, for $t\in\mathbb{C}$ and $n\in\N\cup \{\infty\}$, we define $(t; q)_n := \prod_{i=0}^{n-1}{(1-tq^i)}.$ The study of asymptotic properties of partitions goes back to the seminal work
of Hardy and Ramanujan \cite{HR1918}, who proved that the function $p(n)$ satisfies
\begin{align}\label{HR-Asym}
p(n)\sim \frac{1}{4n\sqrt{3}} e^{\pi\sqrt{\frac{2n}{3}}},\quad\quad \text{as}\quad n\rightarrow \infty,
\end{align}
which gave birth to the \emph{circle method}. This technique was later refined by Rademacher \cite{HRad1937,HRad1943,HRad1973} and quite remarkably he obtained an exact absolutely convergent series for $p(n)$, namely, 
\begin{align}\label{Radamacher}
	p(n)=2\pi \left(\frac{1}{6\sqrt{\frac{2}{3}\left(n-\frac{1}{24}\right)}}\right)^\frac{3}{2}\sum_{k=1}^{\infty}\frac{A_k(n)}{k}I_{\frac{3}{2}}\left(\frac{\pi}{k}\sqrt{\frac{2}{3}\left(n-\frac{1}{24}\right)}\right),
\end{align}
where $I_{v}$ is the modified Bessel function of the first kind and
$A_k(n)$ is a Kloosterman-type sum involving exponential terms, defined as
\begin{align}\label{eqe}
A_{k}(n):=\sum_{\substack{h=0\\\gcd(h,k)=1}}^{k-1}e^{\pi is(h,k)-2\pi i n\frac{h}{k} },
\end{align}
with $s(h, k)$ being the Dedekind sum given by
\begin{align}\label{eq_Dedekind_sum}
	s(h,k)=\sum_{\alpha=1}^{k-1}\frac{\alpha}{k}\left(\frac{h\alpha}{k}-\Bigl\lfloor\frac{h\alpha}{k}\Bigr\rfloor-\frac{1}{2}\right).
\end{align}
Since the publication of Rademacher’s paper \cite{HRad1937}, a number of authors have found series similar to \eqref{Radamacher} for certain restricted partition functions, see, for example, \cite{AGM25,And1966,BM2011,BO2006,Gro1958,Hagis,H70,LKHUA,IJT20,Niven1940,OPR22,Sill2010} and references therein.
In \cite{CL2004}, Corteel and Lovejoy introduced the notion of overpartitions. An overpartition of a nonnegative integer $n$ is a partition of $n$ in which the first occurrence of a part may be overlined. For example, the eight
overpartitions of 3 are $3, ~\overline{3},~ 2 + 1,~ \overline{2} + 1,~ 2 + \overline{1}, ~\overline{2} + \overline{1},~ 1 + 1 + 1,~ \overline{1} + 1 + 1$. In \cite{Love2003}, Lovejoy investigated the function $\overline{A}_\ell(n)$, which counts the number of overpartitions of $n$ into parts not divisible by $\ell$. The relevant 2-regular overpartitions of 3 are $3,~ \overline{3},~1 + 1 + 1,~ \overline{1} + 1 + 1$. In a recent paper \cite{Shen2016}, Shen refers to the overpartitions enumerated by the function $\overline{A}_\ell(n)$ as $\ell$-regular overpartitions. The generating function for $\overline{A}_\ell(n)$ is given by
\begin{align}\label{eq_gf_A}
	\overline{A}_{\ell}(q):=\sum_{n=0}^{\infty}\overline{A}_{\ell}(n)q^{n}=\frac{(-q;q)_\infty(q^\ell;q^\ell)_\infty}{(q;q)_\infty(-q^\ell;q^\ell)_\infty}=\frac{F^{2}(\tau)F(2\ell\tau)}{F(2\tau)F^{2}(\ell\tau)},
\end{align}
where $F(\tau):=\frac{1}{(q;q)_{\infty}}=\frac{1}{(e^{2\pi i\tau};e^{2\pi i\tau})_{\infty}},\quad\text{and}\quad \tau\in \mathbb{H}.$
Recently, Peng, Zhang, and Zhong \cite{Helen25} obtained an asymptotic formula for $\ell$-regular overpartition for $2\leq \ell\leq 9$.
Motivated by their work, we aim to establish a Rademacher-type exact formula for the $\ell$-regular overpartition function for all $\ell\geq3$ squarefree odd integers. To make this precise, we need some more notation. We write $\exp{(x)}$ in place of $e^x$ in some instances. 
For brevity, let $d_k:=\gcd(2,k)$, $\ell_k:=\gcd(\ell,k)$, and
\begin{align*}
\delta_k:=\frac{1}{24}\left(\frac{4}{d_k}-d_k\right)\left(\frac{\ell}{\ell_k}-\ell_k\right),
\end{align*}
for positive integers $k$ and $\ell$. Let $H(r,s)$ be a solution to the congruence $rH(r,s)\equiv-1 \pmod{s}$ and $H(0,1):=0$. We also consider
\begin{align}\label{eqe}
	A_{k,m}(n):=\sum_{\substack{h=0\\\gcd(h,k)=1}}^{k-1}a(m,k)W(h,k)\exp{\left(\frac{\pi i}{\ell k}\left(\ell_k mH(h,k)-2\ell nh\right)\right)},
\end{align}
where the coefficients $a(m,k)$ are defined by \eqref{a(m,k)}, and
\begin{align*}
	W(h,k)
	=\frac{\omega^2(h,k)\omega\left(\frac{2\ell h}{\ell_k},\frac{k}{\ell_k}\right)}{\omega\left(2h,k\right)\omega^2\left(\frac{\ell h}{\ell_k},\frac{k}{\ell_k}\right)},
\end{align*}
with $\omega(h,k)=\exp\left(\pi is(h,k)\right)$.
In terms of the above notation, the following theorem gives the \emph{Rademacher-type formula} for $\overline{A}_{\ell}(n)$.
\begin{theorem}\label{main_theorem}

Let $\ell\geq3$ be a squarefree odd integer. For a positive integer $n$, we have
\begin{align}\label{eq1}
\overline{A}_{\ell}(n)=\sum_{\substack{k=1\\ k \text{ odd}}}^{\infty}\sum_{m=0}^{\lfloor\delta_k\rfloor} \frac{2\pi A_{k,m}(n)}{k}\sqrt{\frac{\ell_k^2(\delta_k-m)}{2\ell^2n}}I_1\left(\frac{4\pi}{k}\sqrt{\frac{\ell_kn(\delta_k-m)}{2\ell}}\right),
\end{align}
where $I_1$ is the modified Bessel function of first kind.
\end{theorem}
\begin{example}
To illustrate the accuracy of our exact formula for $\overline{A}_{\ell}(n)$, we carry out a numerical verification by computing explicit values of $\overline{A}_{\ell}(n)$ for the fixed values of $\ell$. We obtain these values by using the first seven terms of the series expansion in Theorem \ref{main_theorem} and are subsequently compared with their corresponding exact values in Table \ref{main_theorem}.
\begin{table}[h]
\centering
\begin{tabular}{|c|c|c|c|}
\hline
 $\ell$ & $n$ & Exact value of $\overline{A}_{\ell}(n)$ & Value of $\overline{A}_{\ell}(n)$ from our formula \\
 \hline
 7&26&30024&30024.02061\\
 \hline
 11&27&49284&49283.96467\\
 \hline
 15&23&17528&17527.94250\\
 \hline
 21&26&40776&40775.95303\\
 \hline
 23&24&23524&23524.01875\\
 \hline
 29&27&53408&53408.00048\\
 \hline
 31&35&398612&398612.18535\\
 \hline
\end{tabular}
\caption{Numerical verification of Theorem \ref{main_theorem}}
\label{tab1}
\end{table}
\end{example}
As an immediate consequence of Theorem \ref{main_theorem}, we obtain the following corollary which provides an asymptotic formula for $\overline{A}_{\ell}(n)$.
\begin{corollary}\label{cor_1}
Let $\ell\geq3$ be a squarefree odd integer. As $n\rightarrow\infty$, we have
\begin{align}\label{eq2.3}
\overline{A}_{\ell}(n)\sim\frac{1}{2\sqrt{2\ell}}\left(1-\frac{1}{\ell}\right)^{\frac{1}{4}}\left(\frac{1}{n}\right)^{\frac{3}{4}}e^{\pi\sqrt{n\left(1-\frac{1}{\ell}\right)}}.
\end{align}
\end{corollary}
In order to give overpartition analogues of Rogers-Ramanujan type identities for $p(n)$, Andrews \cite{And15} introduced a partition function, which we now call Andrews’ singular overpartition function. For integers $r\geq3$ and $1\leq i\leq \lfloor\frac{r}{2}\rfloor$, let $\overline{C}_{r,i}(n)$ denote the number of Andrews' singular overpartitions of $n$, which are overpartitions of $n$ in which no part is divisible by $r$ and all parts $\equiv\pm i\pmod{r}$ may be overlined. The generating function of $\overline{C}_{r,i}(n)$ is given by
\begin{align*}
\sum_{n=0}^{\infty}\overline{C}_{r,i}(n)q^n=\frac{(q^r;q^r)_{\infty}(-q^i;q^r)_{\infty}(-q^{r-i};q^r)_{\infty}}{(q;q)_{\infty}}.
\end{align*}
Note that $\overline{C}_{3,1}(n)=\overline{A}_{3}(n)$, for all $n\geq0$. Therefore, substituting $\ell=3$ in Theorem \ref{main_theorem}, we get the Rademacher-type formula for $\overline{C}_{3,1}(n)$.
\begin{corollary}\label{cor_2}
For a positive integer $n$, we have
\begin{align*}
\overline{C}_{3,1}(n)=\sum_{\substack{k=1\\ k \text{ odd}}}^{\infty}\sum_{m=0}^{\lfloor\delta_k\rfloor} \frac{2\pi A_{k,m}(n)}{3k}\sqrt{\frac{\ell_k^2(\delta_k-m)}{2n}}I_1\left(\frac{4\pi}{k}\sqrt{\frac{\ell_kn(\delta_k-m)}{6}}\right).
\end{align*}
\end{corollary}

As an application of the exact formula derived above, we investigate the higher-order Tur\'{a}n inequalities for the $\ell$-regular overpartition function. These inequalities play an important role in the study of Maclaurin coefficients of real entire functions belonging to the Laguerre–Pólya class, see, for example, \cite{Dimitrov, szego}. 

A sequence $\{b_n\}_{n\ge 0}$ of real numbers is said to be log-concave, if for all $n \ge 1$,
\begin{align*}
	b_n^2 \ge b_{n-1}b_{n+1}.
\end{align*}
These inequalities are closely connected to Jensen polynomials; see, for example, \cite{CC89,CNV86,CV90}. For a sequence $\{b_n\}_{n\ge 0}$, the Jensen polynomial of degree $d$ and shift $n$, denoted by $J^{d,n}_{b}(X)$, is defined as
\begin{align*}
	J^{d,n}_{b}(X)=\sum_{j=0}^{d}\binom{d}{j}b_{n+j}X^j.
\end{align*}
For $d=2$ and shift $n-1$, the Jensen polynomial $J^{2,n-1}_{b}(X)$ takes the form
\begin{align*}
	J^{2,n-1}_{b}(X)=b_{n-1}+2b_nX+b_{n+1}X^2.
\end{align*}
It follows immediately that the sequence $\{b_n\}_{n\ge 0}$ is log-concave at $n$ if and only if the polynomial $J^{2,n-1}_{b}(X)$ has only real zeros. More generally, the sequence $\{b_n\}_{n\ge 0}$ is said to satisfy the Tur\'{a}n inequality of order $d$ at $n$ if and only if the Jensen polynomial $J^{d,n-1}_{b}(X)$ is hyperbolic, i.e., all of its roots are real. The study of Jensen polynomials has emerged as a powerful tool for understanding the asymptotic behavior and analytic properties of arithmetic functions. For example, Chen, Jia, and Wang \cite{CJW2019} proved the hyperbolicity of the cubic Jensen polynomial $J^{3,n-1}_p(X)$ associated with the partition function $p(n)$ for all $n \ge 94$. They further conjectured that for every integer $d \ge 1$, there exists an integer $N_p(d)$ such that $J^{d,n-1}_{p}(X)$ is hyperbolic for all $n \ge N_p(d)$.
This conjecture was later established by Griffin, Ono, Rolen, and Zagier \cite[Theorem 5]{GORZ19}. Their work not only confirmed the conjecture for the partition function but also demonstrated the hyperbolicity of Jensen polynomials associated with Fourier coefficients of weakly holomorphic modular forms on $\rm{SL}_2(\mathbb{Z})$. Moreover, their results revealed a remarkable connection between Jensen polynomials and the Hermite polynomials $H_d(X)$. In particular, they showed that, under mild analytic assumptions on a positive sequence $\{a(n)\}$, suitably normalized Jensen polynomials converge to Hermite polynomials as $n \to \infty$.
The following theorem, due to Griffin, Ono, Rolen, and Zagier \cite{GORZ19}, provides a general criterion for the convergence of suitably normalized Jensen polynomials to Hermite polynomials.
\begin{theorem}\label{theorem_GORZ}\cite[Theorem 3 and 8]{GORZ19} 
Let $\{a(n)\}, \{b(n)\},$ and $ \{\delta(n)\}$ be sequences of positive real numbers and $\delta(n)\to0$ as $n\to\infty$.  
For integers $r \geq 0$,  $d \geq 1$, suppose that there exist real numbers $c_3(n), c_4(n), \ldots, c_d(n)$, for which 
\begin{align}\label{eq3.2}
\log \left( \frac{a(n+r)}{a(n)} \right)
= b(n)r - \delta(n)^2 r^2 + \sum_{j=3}^{d} c_j(n) r^j + o(\delta(n)^d),
\end{align}
as $n\to\infty$, with $c_j(n) = o(\delta(n)^j)$ for each $3 \leq j \leq d$.  
Then we have  
\begin{align*}
\lim_{n \to \infty} \left( \frac{\delta(n)^{-d}}{a(n)}
J_{a}^{d,n} \!\left( \frac{\delta(n)X - 1}{\exp(a(n))} \right) \right)
= H_d(X).
\end{align*}
\end{theorem}
Since Hermite polynomials have distinct real zeros and real-rootedness is preserved under linear transformations, the theorem implies that the corresponding Jensen polynomials also have distinct real zeros for sufficiently large $n$. Consequently, the higher-order Tur\'{a}n inequalities hold for all sufficiently large $n$.
Over the years, log-concavity and higher-order Tur\'{a}n inequalities have been extensively studied for various partition functions; interested readers can see \cite{AGM25,ChenTalk2010,CJW2019,DeSalvo2015,Dimitrov,LW19,OPR22,Pandey24}. Recently, Peng, Zhang, and Zhong \cite{Helen25} established that when $2\leq \ell \leq 9$, the $\ell$-regular overpartition
function satisfies the conditions of log-concavity and adhere to the third-order Tur\'{a}n inequalities.
Motivated by their work, we prove the following theorem to show the existence of higher order Tur\'{a}n inequalities for the $\ell$-regular overpartition function $\overline{A}_{\ell}(n)$.
\begin{theorem}\label{Theorem2}
Let $\ell\geq3$ be a squarefree odd integer. For any positive integer $d$, $J^{d,n}_{\overline{A}_{\ell}}(X)$ is hyperbolic for all but finitely many values of $n$.
\end{theorem}
To prove these results, we make use of the classical circle method. In Section \ref{Sec2}, we provide the necessary preliminaries required for the proofs. We prove Theorem \ref{main_theorem} in Section \ref{Sec3}. In Section \ref{Sec4}, we deduce Corollary \ref{cor_1} and Theorem \ref{Theorem2} from Theorem \ref{main_theorem}.

\section*{Acknowledgements}
\noindent The second author thanks the INSPIRE faculty research grant and ANRF ECRG (IFA-24MA 204 and ANRF/ECRG/2025/009350/PMS). The authors thank Professor Rupam Barman for reading the initial version of the manuscript and many helpful suggestions. 

\section{Nuts and Bolts}\label{Sec2}
In this section, we present some brief details of Farey fractions and Ford circles, and we state some important results that are useful in the proof of Theorem \ref{main_theorem}.
\subsection{Farey fractions and Ford circles}
In this subsection, we briefly discuss Farey sequences and Ford circles, which play a fundamental role in the circle method. For more on this topic, we refer \cite{Apostal1990} to the reader. For a natural number $N$, the Farey sequence of order $N$, denoted by $F_{N}$ is the sequence of all reduced fractions in the interval $[0,1]$ whose denominators do not exceed $N,$ arranged in increasing order. Thus,
\begin{align*} 
F_{N}:=\bigg\{{\frac{h}{k}\in\mathbb{Q}\cap [0,1]}\bigg\vert \,0\leq h\leq k\leq N,\  \gcd(h,k)=1\bigg\}. 
\end{align*} For each fraction $\frac{h}{k}$ in $F_{N}$, the corresponding Ford circle denoted by $C(h,k)$ is defined by 
$$C(h,k):=\bigg\{ \zeta\in \mathbb{C}:\bigg\vert\,\zeta-\left(\frac{h}{k}+\frac{i}{2k^2}\right)\bigg\vert=\frac{1}{2k^2}\bigg\}.$$ A straightforward computation shows that distinct ford circles never intersect. Furthermore, if $\frac{h_{1}}{k_{1}}<\frac{h}{k}<\frac{h_{2}}{k_{2}}$ are consecutive fractions in a Farey sequence, then $C(h,k)$ is tangent to both $C(h_{1},k_{1})$ and $C(h_{2},k_{2})$. The corresponding points of tangency are
\begin{align}\label{eq_Pt_of_ten}
\tau_1(h,k):=\frac{h}{k}-\frac{k_{1}}{k(k^2+k_{1}^{2})}+\frac{i}{k^2+k_{1}^{2}}\ \text{ and }\ \tau_{2}(h,k):=\frac{h}{k}+\frac{k_{2}}{k(k^{2}+k_{2}^{2})}+\frac{i}{k^{2}+k_{2}^{2}}.
\end{align}
Following Rademacher's refinement of the circle method, for the path of integration, we choose a contour consisting of arcs of Ford circles associated with the Farey sequence of order $N$. By decomposition the contour into combinations arising from individual Farey fractions, one can obtain that this decomposition isolates the contribution from each rational cusp and facilitates the analysis of the generating function in its vicinity. Consequently, global integrals can be expressed as a sum of local contributions, leading to an exact formula for $\overline{A}_{\ell}(n).$ 
\subsection{Transformation Formula}
In this subsection, we establish the transformation formula for $F(\alpha\tau)$ for $\alpha\in\{1,2,\ell,2\ell\}$ and $\tau=\frac{h}{k}+\frac{iz}{k^2}$, with $\frac{h}{k}\in\mathbb{Q}$ and $\Re(z)>0$. This will be used to find a transformation formula for $\overline{A}_{\ell}\left(\exp\left(2\pi i\left(\frac{h}{k}+\frac{iz}{k^2}\right)\right)\right)$.
Recall the Dedekind eta function
\begin{align*}
\eta(\tau)=e^{\pi i\tau/12}\prod_{j=0}^{\infty}\left(1-e^{2\pi ij\tau}\right),\ \ \tau\in\mathbb{H}\,,
\end{align*}
and its transformation formula \cite[Theorem 3.4]{Apostal1990}
\begin{align}\label{eq_eta_TF}
\eta\left(\frac{a\tau+b}{c\tau+d}\right)=\exp\left(\frac{\pi i}{12}\left(\frac{a+d}{c}+12s(-d,c)\right)
\right)\sqrt{\frac{c\tau+d}{i}}\eta(\tau),
\end{align}
for any $\tau\in\mathbb{H}$ and $\begin{bmatrix}
a & b \\ c & d 
\end{bmatrix}\in \text{SL}_2(\mathbb{Z})$, with $c>0$.
Here, $s(h,k)$ is the Dedekind sum defined in \eqref{eq_Dedekind_sum}. Note that
\begin{align}\label{eq_F_eta}
F(\tau)=e^{\pi i\tau/12}\eta^{-1}(\tau).
\end{align}
Substituting $\eta^{-1}(\tau)$ from \eqref{eq_eta_TF} in \eqref{eq_F_eta}, we get
\begin{align}\label{eq3.a}
F(\tau)&= \exp\left(\frac{\pi i}{12}\left(\tau+\frac{a+d}{c}+12s(-d,c)\right)\right)\sqrt{\frac{c\tau+d}{i}}\eta^{-1}\left(\frac{a\tau+b}{c\tau+d}\right)\nonumber\\
&=\exp\left(\frac{\pi i}{12}\left(\tau-\frac{a\tau+b}{c\tau+d}+\frac{a+d}{c}+12s(-d,c)\right)\right)\sqrt{\frac{c\tau+d}{i}}F\left(\frac{a\tau+b}{c\tau+d}\right).
\end{align}
For a positive integer $k$ and $\alpha\in\{1,2,\ell,2\ell\}$, let $\alpha_k:=\gcd(\alpha,k)$. Thus, $$(\alpha,\alpha_k)\in\{(1,1),(2,d_k),(\ell,\ell_k),(2\ell,d_k\ell_k)\}.$$ Here $\gcd(2\ell,k)=d_k\ell_k$ as $\ell$ is a squarefree odd integer. For a transformation formula of $F(\alpha\tau)$, we replace $\tau$ in \eqref{eq3.a} with $\alpha\left(\frac{h}{k}+\frac{iz}{k^2}\right)$, and set $a=H\left(\frac{h\alpha}{\alpha_k},\frac{k}{\alpha_k}\right)$, $c=\frac{k}{\alpha_k}$, $d=-\frac{h\alpha}{\alpha_k}$, and $b=\frac{ad-1}{c}$ to get
\begin{align}\label{eq3.b}
F\left(\alpha\left(\frac{h}{k}+\frac{iz}{k^2}\right)\right)=&\exp\left(\frac{\pi\alpha_k^2}{12\alpha z}-\frac{\pi\alpha z}{12k^2}+\pi i s\left(\frac{h\alpha}{\alpha_k},\frac{k}{\alpha_k}\right)\right)
\sqrt{\frac{\alpha z}{k\alpha_k}}\\
&\times F\left(\frac{H\left(\frac{h\alpha}{\alpha_k},\frac{k}{\alpha_k}\right)}{k/\alpha_k}+\frac{i\alpha_k^2}{\alpha z}\right).\nonumber
\end{align}
Note that $\tau=\alpha\left(\frac{h}{k}+\frac{iz}{k^2}\right)\in\mathbb{H}$ as $\Re(z)>0$, and $\begin{bmatrix}
a & b \\ c & d 
\end{bmatrix}\in \text{SL}_2(\mathbb{Z})$ with $c>0$.
Substituting $\tau$ with $\frac{h}{k}+\frac{iz}{k^2}$ in \eqref{eq_gf_A} and using \eqref{eq3.b} for $\alpha\in\{1,2,\ell,2\ell\}$, we get
\begin{align}\label{eq3.c}
\overline{A}_{\ell}\left(\exp\left(2\pi i\left(\frac{h}{k}+\frac{iz}{k^2}\right)\right)\right)=&\frac{F^{2}\left(\frac{h}{k}+\frac{iz}{k^2}\right)F\left(2\ell\left(\frac{h}{k}+\frac{iz}{k^2}\right)\right)}{F\left(2\left(\frac{h}{k}+\frac{iz}{k^2}\right)\right)F^2\left(\ell\left(\frac{h}{k}+\frac{iz}{k^2}\right)\right)}\nonumber\\
=&W(h,k)\sqrt{\frac{\ell_k}{\ell}}\exp\left(\frac{\pi d_k\ell_k\delta_k}{\ell z}\right)\\
&\times
\frac{F^2\left(\frac{H\left(h,k\right)}{k}+\frac{i}{z}\right)F\left(\frac{H\left(\frac{2\ell h}{d_k\ell_k},\frac{k}{d_k\ell_k}\right)}{k/(d_k\ell_k)}+\frac{id_k^2\ell_k^2}{2\ell z}\right)}{F\left(\frac{H\left(\frac{2h}{d_k},\frac{k}{d_k}\right)}{k/d_k}+\frac{id_k^2}{2z}\right)F^2\left(\frac{H\left(\frac{\ell h}{\ell_k},\frac{k}{\ell_k}\right)}{k/\ell_k}+\frac{i\ell_k^2}{\ell z}\right)},\nonumber
\end{align}
where $\delta_k:=\frac{1}{24}\left(\frac{4}{d_k}-d_k\right)\left(\frac{\ell}{\ell_k}-\ell_k\right)$, and
\begin{align}\label{eq3.d}
W(h,k):=\frac{\omega^2\left(h,k\right)\omega\left(\frac{2\ell h}{d_k\ell_k},\frac{k}{d_k\ell_k}\right)}{\omega\left(\frac{2 h}{d_k},\frac{k}{d_k}\right)\omega^2\left(\frac{\ell h}{\ell_k},\frac{k}{\ell_k}\right)}.
\end{align}
\par Next, we deal with the terms $F\left(\frac{\alpha_k}{k}H\left(\frac{\alpha h}{\alpha_k},\frac{k}{\alpha_k}\right)+\frac{i\alpha_k^2}{\alpha z}\right)$ on the R.H.S of \eqref{eq3.c}. Suppose $\gcd(h,k)=1$. From
\begin{align*}
\frac{\alpha h}{\alpha_k}H\left(\frac{\alpha h}{\alpha_k},\frac{k}{\alpha_k}\right)\equiv-1
\pmod[\bigg]{\frac{k}{\alpha_k}}
\end{align*}
and 
\begin{align}\label{eq3e}
hH(h,k)\equiv-1\pmod{k},
\end{align}
we get
\begin{align}\label{eq3e1}
\alpha H\left(\frac{\alpha h}{\alpha_k},\frac{k}{\alpha_k}\right)\equiv\alpha_k H(h,k)\pmod{k}.
\end{align}
For any $\alpha\in\{1,2,\ell,2\ell\}$, we have $\gcd\left(\frac{\alpha}{\alpha_k},k\right)=1$, as $\ell$ is a squarefree odd integer. This implies that $\gcd\left(\frac{\alpha h}{\alpha_k},k\right)=1$, which further implies that the congruence
\begin{align}\label{eq3f}
h\left(\frac{\alpha}{\alpha_k}H'(h,k)\right)\equiv-1\pmod{k}
\end{align}
has a solution. From \eqref{eq3e} and \eqref{eq3f}, we observe that $H(h,k)$ can be taken as a multiple of $\frac{\alpha}{\alpha_k}$. In that case, we have
\begin{align}\label{eq3g}
\alpha H\left(\frac{\alpha h}{\alpha_k},\frac{k}{\alpha_k}\right)\equiv\alpha_k H(h,k)\pmod{\alpha}.
\end{align}
From \eqref{eq3e1}, \eqref{eq3g}, the fact that $\gcd\left(\frac{\alpha}{\alpha_k},k\right)=1$, and the \emph{Chinese Remainder Theorem,} we obtain
\begin{align}\label{eq3h}
\alpha H\left(\frac{\alpha h}{\alpha_k},\frac{k}{\alpha_k}\right)\equiv\alpha_k H(h,k)\pmod[\bigg]{\frac{\alpha k}{\alpha_k}}.
\end{align}
It follows from \eqref{eq3h} and the periodicity of $F(\tau)$, i.e., $F(\tau+1)=F(\tau)$ that 
\begin{align}\label{eq3i}
F\left(\frac{\alpha_k}{k}H\left(\frac{\alpha h}{\alpha_k},\frac{k}{\alpha_k}\right)+\frac{i\alpha_k^2}{\alpha z}\right)=F\left(\frac{\alpha_k^2H(h,k)}{\alpha k}+\frac{i\alpha_k^2}{\alpha z}\right).
\end{align}
If we set
\begin{align*}
\nu_z:=\frac{d_k\ell_k}{2\ell}\left(\frac{H(h,k)}{k}+\frac{i}{z}\right),
\end{align*}
then, with the application of \eqref{eq3i} for $(\alpha,\alpha_k)\in\{(1,1),(2,d_k),(\ell,\ell_k),(2\ell,d_k\ell_k)\}$, the
last factor on the R.H.S. of \eqref{eq3.c} becomes
\begin{align}\label{eq3j}
\frac{F^2\left(\frac{2\ell}{d_k\ell_k}\nu_z\right)F\left(d_k\ell_k\nu_z\right)}{F\left(\frac{d_k\ell}{\ell_k}\nu_z\right)F^2\left(\frac{2\ell_k}{d_k}\nu_z\right)}=:\mathcal{A}_k(\nu_z).
\end{align}
Here, we write
\begin{align}\label{eq3k}
\mathcal{A}_k(\tau)=\sum_{m=0}^{\infty}a(m,k)e^{2\pi im\tau},
\end{align}
for some coefficients $a(m,k)$. With this, \eqref{eq3.c} takes the following form
\begin{align}\label{eq3l}
\overline{A}_{\ell}\left(\exp\left(2\pi i\left(\frac{h}{k}+\frac{iz}{k^2}\right)\right)\right)=W(h,k)\sqrt{\frac{\ell_k}{\ell}}\exp\left(\frac{\pi d_k\ell_k\delta_k}{\ell z}\right)\mathcal{A}_k(\nu_z).
\end{align}
\subsection{Estimate for the coefficients $a(m,k)$} Here we give an estimate for the growth rate of the coefficients $a(m,k)$ appearing in \eqref{eq3k}. Before moving further, we recall the Euler's pentagonal number theorem \cite[Corollary 1.7]{And1998}
\begin{align}\label{eq_epnt}
(q;q)_{\infty}=\sum_{j=-\infty}^{\infty}(-1)^jq^{j(3j-1)/2}.
\end{align}

From \eqref{eq3j} and \eqref{eq3k}, we arrive at the following:
\begin{align}\label{a(m,k)}
\sum_{m=0}^{\infty}a(m,k)q^m=\frac{(q^{s};q^{s})_{\infty}(q^{t};q^{t})_{\infty}^{2}}{(q^u;q^u)_{\infty}^{2}(q^{v};q^{v})_{\infty}},
\end{align}
where $s=\frac{d_k\ell}{\ell_k}$, $t=\frac{2\ell_k}{d_k}$, $u=\frac{2\ell}{d_k\ell_k}$, and $v=d_k\ell_k$. Here, for simplicity, we abuse the notation by taking $q=e^{2\pi i\nu_z}$.
From the generating function of $p(n)$, we get
\begin{align*}
\frac{1}{(q^r;q^r)_{\infty}}=\sum_{n=0}^{\infty}p(n)q^{rn},
\end{align*}
which suggests that the coefficient $p_{uuv}(n)$ in 
\begin{align*}
\frac{1}{(q^u;q^u)_{\infty}^2(q^v;q^v)_{\infty}}=:\sum_{n=0}^{\infty}p_{uuv}(n)q^n,
\end{align*}
have the following estimate: $p_{uuv}(n)\ll (m+1)^2(p(m))^3$. It is easy to see from \eqref{eq_epnt} that the coefficients $p_{stt}(n)$ in 
\begin{align*}
(q^s;q^s)_{\infty}(q^t;q^t)_{\infty}^2=:\sum_{n=0}^{\infty}p_{stt}(n)q^n,
\end{align*}
have the following estimate: $p_{stt}(n)\ll \sqrt{n}$. Finally, using \eqref{HR-Asym}, we conclude that
\begin{align*}
a(m,k)\ll m\cdot |p_{stt}(m)|\cdot |p_{uuv}(m)|\ll m\cdot m^2 (p(m))^3\cdot \sqrt{m}\ll \sqrt{m}e^{3\pi\sqrt{2m/3}}\ll e^{3\pi\sqrt{m}}.
\end{align*}
Therefore, 
\begin{align}\label{eq_a_m,k_bound}
a(m,k)\ll e^{3\pi\sqrt{m}},
\end{align}
where the implicit constant does not depend on $k$.
\subsection{An expression for $W(h,k)$:} In this subsection, we calculate $W(h,k)$ arising in the transformation formula for $\overline{A}_{\ell}\left(\exp\left(2\pi i\left(\frac{h}{k}+\frac{iz}{k^2}\right)\right)\right)$ (see, \eqref{eq3l}) defined by \eqref{eq3.d}.
\par For a positive integer $k$ and $\alpha\in\{1,2,\ell,2\ell\}$, let $\alpha_k:=\gcd(\alpha,k)$. Let
\begin{equation}\label{eq3a0}
\omega_{\alpha}(h,k):=\exp\left(\frac{-2\pi i(k^{2}-\alpha_k^{2})}{24k\alpha_k^3}\left(2h\alpha\alpha_k+(\alpha^{2}h^{2}-\alpha_k^{2})H\left(\frac{\alpha h}{\alpha_k},\frac{k}{\alpha_k}\right)\right)\right)
\end{equation} 
and
\begin{align}\label{very_new_eq}
W^*(h,k):=\frac{\omega_{1}^{2}(h,k)\omega_{2\ell}(h,k)}{\omega_{2}(h,k)\omega_{\ell}^{2}(h,k)}.
\end{align}
From \cite[Eq. (2.3) and (2.4)]{Niven1940}, we have
\begin{equation}\label{eq3a1}
\omega\left(\frac{\alpha h}{\alpha_k},\frac{k}{\alpha_k}\right)=
\begin{cases}
\left(\frac{-h\alpha/\alpha_k}{k/\alpha_k}\right)\exp{\left(\frac{-\pi i}{4}(k-1)\right)}\omega_{\alpha}(h,k), & \text{if } k \text{ is odd},\\\\
\left(\frac{-k/\alpha_k}{h\alpha/\alpha_k}\right)\exp{\left(\frac{-\pi i}{4}\left(2-\frac{h\alpha}{\alpha_k^{2}}(k+\alpha_k)\right)\right)}\omega_{\alpha}(h,k),& \text{if } k \text{ is even}.
\end{cases}
\end{equation}
Here, $\left(\frac{\cdot}{\cdot}\right)$ is the Jacobi symbol.
Then using \eqref{eq3a1} in \eqref{eq3.d}, we obtain
\begin{align}\label{eq3a5}
W(h,k)=
\begin{cases}       
W^*(h,k)\left(\frac{\ell/\ell_{k}}{k/\ell_{k}}\right)\left(\frac{-2h}{\ell_{k}}\right), &  \text{ if }k\text{ is odd},\\\\
W^*(h,k)\left(\frac{-k/2\ell_{k}}{\ell/\ell_{k}}\right)\left(\frac{\ell_{k}}{h}\right)\exp{\left(\frac{-\pi ih}{4}\left(\frac{3k}{2}\left(\frac{\ell}{\ell_{k}^{2}}-1\right)+\left(\frac{\ell}{\ell_{k}}-1\right)\right)\right)}, &  \text{ if }k\text{ is even}.
\end{cases}
\end{align}  Now, set $24\cdot2\ell=PQ$, where $P$ and $Q$ are taken in such a way that $P$ is the largest divisor of $24\cdot 2 \ell$ with $\gcd(P,k)=1$. Also, consider $P'$ such that $P'P\equiv 1\pmod {Qk}$. Since $P$ is the largest divisor that doesn't have a prime factor of $k$, if a prime $p$ divides $Q$ then $p$ also divides $k$. This implies that $\gcd(h,k)=1$ if and only if $\gcd(h,Qk)=1$. It follows that there exists $H(h,Qk)$ such that $P$ divides $H(h,Qk)$, i.e., $hH(h,Qk)\equiv -1 \pmod {Qk}$ and $P\mid H(h,Qk)$ (similar to the argument followed by \eqref{eq3e1} since $\gcd(Ph,Qk)=1$). Note that $\gcd\left(\frac{h\alpha}{\alpha_k},Qk\right)=1$. Therefore, there exists $H\left(\frac{h\alpha}{\alpha_k},\frac{Qk}{\alpha_k}\right)$ that is divisible by $P$. Next, from
\begin{align*}
\frac{\alpha h}{\alpha_k}H\left(\frac{\alpha h}{\alpha_k},\frac{Qk}{\alpha_k}\right)\equiv-1\pmod[\bigg]{\frac{Qk}{\alpha_k}}
\end{align*}
and
\begin{align*}
hH\left(h,Qk\right)\equiv-1\pmod[\bigg]{\frac{Qk}{\alpha_k}},
\end{align*}
we have
\begin{align}\label{eq3b1}
\frac{1}{P}H\left(\frac{\alpha h}{\alpha_k},\frac{Qk}{\alpha_k}\right)\equiv\frac{P'\alpha_k}{\alpha}H\left(h,Qk\right)\ \pmod[\bigg]{\frac{Qk}{\alpha_k}}.
\end{align}
Here both the L.H.S. and the R.H.S. of \eqref{eq3b1} are integers as $\frac{\alpha}{\alpha_k}\mid P$, $P\mid H(h,Qk)$, and $P\mid H\left(\frac{h\alpha}{\alpha_k},\frac{Qk}{\alpha_k}\right)$. Making use of \eqref{eq3b1}, and the facts that $24=PQ/(2\ell)$, $\alpha_k\mid 2\ell$ and $\alpha_k^2\mid (k^2-\alpha_k^2)$ in \eqref{eq3a0} yields
\begin{align*}
\omega_{\alpha}(h,k)=&\exp\left(\frac{-\pi i(k^{2}-\alpha_k^{2})h\alpha}{6k\alpha_k^2}\right)\exp\left(\frac{-2\pi i}{24k\alpha_k^3}(k^{2}-\alpha_k^{2})(\alpha^{2}h^{2}-\alpha_k^{2})H\left(\frac{\alpha h}{\alpha_k},\frac{k}{\alpha_k}\right)\right)\\
=&\exp\left(\frac{-\pi i(k^{2}-\alpha_k^{2})h\alpha}{6k\alpha_k^2}\right)\\
&\times\exp\left(-2\pi i\left(\frac{1}{Qk/\alpha_k}\cdot\frac{2\ell}{\alpha_k}\cdot\frac{(k^2-\alpha_k^2)(\alpha^2h^2-\alpha_k^2)}{\alpha_k^2}\cdot\frac{1}{P\alpha_k}H\left(\frac{\alpha h}{\alpha_k},\frac{k}{\alpha_k}\right)\right)\right)\\
=&\exp\left(\frac{-\pi i(k^{2}-\alpha_k^{2})h\alpha}{6k\alpha_k^2}\right)\exp\left(\frac{-4\pi i\ell P'}{Qk\alpha\alpha_k^2}(k^2-\alpha_k^2)(\alpha^2h^2-\alpha_k^2)H\left(h,Qk\right)\right)\\
=&\exp\left(\frac{-\pi i(k^{2}-\alpha_k^{2})h\alpha}{6k\alpha_k^2}\right)\exp\left(-4\pi i\ell\alpha P'\cdot\frac{(k^2-\alpha_k^2)}{\alpha_k^2}\cdot\frac{1}{Qk} \left(h^2H\left(h,Qk\right)\right)\right)\\
&\times\exp\left(\frac{4\pi i\ell P'}{Qk\alpha}(k^2-\alpha_k^2)H\left(h,Qk\right)\right),
\end{align*}
where the second exponential term is equal to 1, because $\alpha_k^2\mid (k^2-\alpha_k^2)$ and $h^2H(h,Qk)\equiv-h\pmod{Qk}$. Therefore,
\begin{align}\label{eq3a3}
\omega_{\alpha}(h,k)
=\exp\left(\frac{-\pi i(k^{2}-\alpha_k^{2})h\alpha}{6k\alpha_k^2}\right)
\exp\left(\frac{4\pi i\ell P'}{Qk\alpha}(k^2-\alpha_k^2)H\left(h,Qk\right)\right).
\end{align}
Employing \eqref{eq3a3} for $\alpha\in\{1,2,\ell,2\ell\}$ in \eqref{very_new_eq}, we get
\begin{align}\label{eq3a4}
W^*(h,k)=\rho(h,k)\exp\left(\frac{-\pi ihk}{3}\left(1+\frac{\ell}{d_k^2\ell_k^2}-\frac{1}{d_k^2}-\frac{\ell}{\ell_k^2}\right)\right),
\end{align}
where
\begin{align*}
\rho(h,k)=\exp\left(\frac{48\pi iP'H(h,Qk)}{Qk}\left(\frac{k^2(\ell-1)}{8}-d_k\ell_k\delta_k\right)\right).  
\end{align*}
Finally, substituting \eqref{eq3a4} into \eqref{eq3a5}, we get the required expression for $W(h,k)$:
\begin{align}\label{eq3a6}
W(h,k)=
\begin{cases}       
\left(\frac{\ell/\ell_{k}}{k/\ell_{k}}\right)\left(\frac{-2h}{\ell_{k}}\right)\rho(h,k), &  \text{if }k\text{ is odd},\\\\
\left(\frac{-k/2\ell_{k}}{\ell/\ell_{k}}\right)\left(\frac{\ell_{k}}{h}\right)\exp{\left(\frac{-\pi ih}{4}\left(\frac{k}{2}\left(\frac{\ell}{\ell_{k}^{2}}-1\right)+\left(\frac{\ell}{\ell_{k}}-1\right)\right)\right)}\rho(h,k), &  \text{if }k\text{ is even}.
\end{cases}
\end{align}
\subsection{Other required results} In this subsection, we recall some results that are important for deriving the Rademacher-type formula and to prove Tur\'{a}n inequalities of higher order for $\overline{A}_{\ell}(n)$.
\par For integers $\alpha$, $\beta$, and $\gamma$ with $\gamma>0$, the Kloosterman sum is defined by
\begin{align}\label{eq_KS}
\mathcal{K}(\alpha,\beta;\gamma):=\sum_{\substack{1\leq j\leq\gamma\\ \gcd(j,\gamma)=1}}\exp{\left(\frac{2\pi i(\alpha j+\beta j')}{\gamma}\right)},
\end{align}
where $j'$ is the multiplicative inverse of $j$ modulo $\gamma$. The Weil's bound for Kloosterman sum is given in the following lemma (see, \cite[Corollary 11.12]{IK2004}).
\begin{lemma}\label{lemma_Weil}
We have
\begin{align*}
\mathcal{K}(\alpha,\beta;\gamma)\ll\left(\gcd (\alpha,\beta,\gamma)\right)^{1/2}\gamma^{1/2+\varepsilon}.
\end{align*}
\end{lemma}
For $s>0$ and $\Re(\nu)>0$, the modified Bessel function of the first kind has the following integral representation:
\begin{align}\label{eq_BF_1}
I_{\nu}(\zeta)=\frac{(\zeta/2)^{\nu}}{2\pi i}\int_{s-i\infty}^{s+i\infty}u^{-\nu-1}\exp\left(u+\frac{\zeta^{2}}{4u}\right)du.
\end{align}
Next, we state an asymptotic result for $I_{\nu}(x)$, see, for example, \cite{Watson44}.
\begin{lemma}
As $x\rightarrow\infty$, we have
\begin{align}\label{eq_I_asym}
I_{\nu}(x)\sim\frac{e^x}{\sqrt{2\pi x}}.
\end{align}
\end{lemma}
\section{Proof of Theorem \ref{main_theorem}}\label{Sec3}
In this section, we prove Theorem \ref{main_theorem}, that is, we find Rademacher-type formula for $\ell$-regular overpartitions. 
\par Let $\ell > 2,$ be a squarefree odd integer and
$\overline{A}_{\ell}(n)$ be the number of $\ell$ regular overpartitions of $n$. Using the \emph{Cauchy Residue Theorem}, we get
\begin{align*}
\overline{A}_{\ell}(n) &=\frac{1}{2\pi i}\int_{C}\frac{\overline{A}_{\ell}(x)}{x^{n+1}}dx,
\end{align*}
where the integration path $C$ is the positively oriented circle centered at the origin with radius $e^{-2\pi}$. Using the parameterization $x=e^{2\pi i\tau}$ with $\tau\in\mathbb{H}$, we obtain
    \begin{align}\label{eqa}
    \overline{A}_{\ell}(n)&=\int_{i}^{i+1} \overline{A}_{\ell}(e^{2\pi i\tau})e^{-2\pi i n\tau}d\tau.
    \end{align}
Since the integrand is analytic, any contour joining $i$ to $i+1$ may be employed, here we adopt a standard Farey dissection contour. Fix a natural number $N$, consider the ford circle $C(h,k)$ associated with the farey fraction $\frac{h}{k}$ of order $N$. For each $\frac{h}{k}$ , let $\gamma(h,k)$ denote the upper arc of the Ford circle $C(h,k)$ connecting the tangency points $\tau_1(h,k)$ and $\tau_2(h,k)$, defined by \eqref{eq_Pt_of_ten}.
We change the path of integration in \eqref{eqa} to be the union of these upper arcs $\gamma(h,k)$ to have
\begin{align}\label{eq_897}
\overline{A}_{\ell}(n) &=\sum_{k=1}^{N}\sum_{\substack{h=0\\\gcd(h,k)=1}}^{k-1}\int_{\gamma(h,k)}\overline{A}_{\ell}(e^{2\pi i\tau})e^{-2\pi in\tau}d\tau.
\end{align}
Next, we employ $\tau=\frac{h}{k}+\frac{i\zeta}{k^2}$, which changes the limit of integration from the arc $\gamma(h,k)$ of the circle $C(h,k)$ to the circle $\lvert \zeta-\frac{1}{2}\rvert=\frac{1}{2}$ with $\zeta$ varying from $\zeta_{1}(h,k)$ to $\zeta_{2}(h,k)$, where
$$\zeta_{1}(h,k):=\frac{k^2+ ikk_{1}}{k^2+k_{1}^2}\quad \text{and} \quad \zeta_{2}(h,k):=\frac{k^2-ikk_{2}}{k^2+k_{2}^{2}}.$$
Then, \eqref{eq_897} becomes
\begin{align}
\overline{A}_{\ell}(n) &=\sum_{\substack{0\leq h<k\leq N\\\gcd(h,k)=1}}\frac{i}{k^{2}}\int_{\zeta_{1}(h,k)}^{\zeta_{2}(h,k)}\overline{A}_{\ell}\left(\exp\left(2\pi i\left(\frac{h}{k}+\frac{i\zeta}{k^2}\right)\right)\right)\exp{\left(-2\pi in\left(\frac{h}{k}+\frac{i\zeta}{k^2}\right)\right)}d\zeta\nonumber.
\end{align} 
Using transformation for $\overline{A}_{\ell}\bigg(\exp\left(2\pi i\left(\frac{h}{k}+\frac{i\zeta}{k^2}\right)\right)\bigg)$ from $\eqref{eq3l}$, we get
\begin{align*}
\overline{A}_{\ell}(n) =&\sum_{\substack{0\leq h<k\leq N\\\gcd(h,k)=1}}\frac{i}{k^{2}}W(h,k)\sqrt{\frac{\ell_{k}}{\ell}}e^{-2\pi in\frac{h}{k}}\int_{\zeta_{1}(h,k)}^{\zeta_{2}(h,k)} \exp{\left(\frac{\pi}{\zeta} \frac{d_{k}\ell_{k}\delta_{k}}{\ell}+\frac{2\pi n\zeta}{k^{2}}\right)}\mathcal{A}_{k}(\nu_{\zeta})d{\zeta}\\
=&\sum_{k=1}^{N}\frac{i}{k^{2}}\sqrt{\frac{\ell_{k}}{\ell}}\sum_{m=0}^{\infty}a(m,k)\sum_{\substack{0\leq h\leq k-1\\\gcd(h,k)=1}}W(h,k)\exp{\left(\frac{\pi i}{\ell k}(d_{k}\ell_{k}mH(h,k)-2\ell nh)\right)}\\
&\times\int_{\zeta_{1}(h,k)}^{\zeta_{2}(h,k)} \exp{\left(\frac{\pi d_{k}\ell_{k}(\delta_{k}-m)}{\ell\zeta}+\frac{2\pi n\zeta}{k^{2}}\right)}d{\zeta}.
\end{align*}
For the sake of ease, we define
\begin{align}\label{eq_defn_psi}
\psi_{m,k}(\zeta):= \exp{\left(\frac{\pi d_{k}\ell_{k}(\delta_{k}-m)}{\ell\zeta}+\frac{2\pi n\zeta}{k^{2}}\right)}.
\end{align}
Therefore, we have
\begin{align}\label{eqb}
\overline{A}_{\ell}(n)=&\sum_{k=1}^{N}\frac{i}{k^{2}}\sqrt{\frac{\ell_{k}}{\ell}}\sum_{m=0}^{\infty}a(m,k)\sum_{\substack{0\leq h \leq k-1\\\gcd(h,k)=1}}W(h,k) \exp\left(\frac{\pi i}{\ell k}(d_{k}\ell_{k}mH(h,k)-2\ell nh)\right)\\
&\times\int_{\zeta_{1}(h,k)}^{\zeta_{2}(h,k)}\psi_{m,k}(\zeta)d{\zeta}.\nonumber
\end{align}
Let us split the sum over $m$ into two parts, $m<\delta_{k} \,\text{and}\,m\geq \delta_{k}$, and write, respectively,
\label{4.4}
\begin{align}\label{eq_new4}
\overline{A}_{\ell}(n)=I_{1}(n;N)+I_{2}(n;N),
\end{align}
for the decomposition arising from \eqref{eqb}. Here, $I_{2}(n;N)$ is an error term that makes a negligible contribution to the sum, and $I_1(n;N)$ is the main term. We separately handle these terms in the next subsections.
\subsection{Estimates for $I_{2}(n;N)$}
To facilitate analysis, we partition the integration path connecting $\zeta_{1}(h,k)$ and $\zeta_{2}(h,k)$ into the three arcs $[\zeta(k_{1}),\zeta(N)]$, $[\zeta(-N),\zeta(N)]$, and $[\zeta(-N),\zeta(-k_{2})]$, where 
\begin{align*}
\zeta(t):=\frac{k^{2}}{k^{2}+t^{2}}+\frac{ikt}{k^{2}+t^2}.
\end{align*}
A direct calculation shows that all the points $\zeta(t)$ lie on the circle $|\zeta-\frac{1}{2}|=\frac{1}{2}$. We have three consecutive Farey fractions $\frac{h_1}{k_1}$, $\frac{h}{k}$, and $\frac{h_2}{k_2}$ of order $N$. We know that the sum of the denominators of two consecutive Farey fractions is greater than the order, i.e., $k+k_1>N$ and $k+k_2>N$. This implies that $k_1\geq N+1-k$ and $-k_2\leq k-N-1$, and we replace the arcs $[\zeta(k_{1}),\zeta(N)]$ and $[\zeta(-N),\zeta(-k_2)]$ by the union of arcs of the form $[\zeta(t),\zeta(t+1)]$ with $N+1-k\leq t\leq N-1$ and $-N\leq t\leq k-N-2$, respectively. Also, we have $hk_1-h_1k=1=h_2k-hk_2$, which implies that $hH(h,k)\equiv hk_2\equiv-hk_1\equiv-1\pmod{k}$. But $\gcd(h,k)=1$ and therefore $H(h,k)\equiv k_2\equiv-k_1\pmod{k}$. Thus, writing $-k_2\leq t\leq k_1-1$ is the same as writing $H(h,k)\in I_t$, for some interval $I_t$ modulo $k$, and this enables us to interchange the order of summation on $h$ and the integration in \eqref{eqb}. With all the discussion above, for $m\geq\delta_k$, \eqref{eqb} can be written as 
\begin{align}\label{eq_new0}
I_{2}(n;N) =\sum_{k=1}^{N}\frac{i}{k^{2}}\sqrt{\frac{\ell_{k}}{\ell}}\sum_{m\geq\delta_{k}}a(m,k)\left(T_{1}(m,k)+T_{2}(m,k)+T_{3}(m,k)\right),
\end{align}
where 
\begin{align*}
T_{1}=T_{1}(m,k)&=\sum_{t=N+1-k}^{N-1}\int_{\zeta(t)}^{\zeta(t+1)}\psi_{m,k}(\zeta)\Phi(k,m,t)d{\zeta},\\
T_{2}=T_{2}(m,k)&=\sum_{t=-N}^{k-N-2}\int_{\zeta(t)}^{\zeta(t+1)}\psi_{m,k}(\zeta)\Phi(k,m,t)d{\zeta},\\
T_{3}=T_{3}(m,k)&=\int_{\zeta(-N)}^{\zeta(N)}\psi_{m,k}(\zeta)\Phi(k,m,t)d{\zeta},
\end{align*}
with
\begin{align}\label{eq_defn_Phi}
\Phi(k,m,t):=\sum_{\substack{0\leq h \leq k-1\\\gcd(h,k)=1\\H(h,k)\in I_t}}W(h,k) \exp{\left(\frac{\pi i}{\ell k}(d_{k}\ell_{k}mH(h,k)-2\ell nh)\right)}.
\end{align}
\par In order to derive estimates for $T_{1}$, $T_{2}$, and $T_{3}$, we reformulate $\Phi(k,m,t)$ using Kloosterman sums, and using Weil's bound for Kloosterman sums, we give the following estimates for $\Phi(k,m,t)$.
\begin{lemma}\label{lemma_Phi}
We have
\begin{align*}
\Phi(k,m,t)\ll_{\ell,n,\varepsilon}k^{\frac{1}{2}+\varepsilon},
\end{align*}
where the implicit constant is independent of $k$.
\end{lemma}
\begin{proof}
In the proof, we use Lemma 1 of \cite{LP2012}. On the similar lines of \cite[Lemma 1]{LP2012}, with $r$ and $s$ replaced by $2$ and $\ell$, respectively, one can prove the following: For all $k$ and $m$, there exists $j\in\{ 0,\ldots,k-1\}$ such that for every $t$, we have
\begin{align*}
\Phi(k,m,t)\ll (1+\log{k})\Bigg\lvert
\sum_{\substack{0\leq h \leq k-1\\\gcd(h,k)=1}}W(h,k) \exp{\left(\frac{\pi i}{\ell k}\left((d_{k}\ell_{k}m+2\ell j)H(h,k)-2\ell nh)\right)\right)}\Bigg\rvert.
\end{align*}
Let us denote the sum on the R.H.S. of the above expression by $\phi(k,m)$ to have
\begin{align}\label{eq_Phi}
\Phi(k,m,t)\ll (1+\log{k})\phi(k,m).
\end{align}
Next, we write some observations. First, $W(h+ck,k)=W(h,k)$, for all $c\in\mathbb{Z}$. Second, $48\ell=PQ$ and $\gcd(h,k)=1$ if and only if $\gcd(h,Qk)=1$. Third,
\begin{align*}
H(h,Qk)\equiv-h^{-1}\equiv PH(Ph,Qk)\pmod{Qk}
\end{align*}
or
\begin{align*}
H(Ph,Qk)\equiv P'H(h,Qk)\pmod{Qk}.
\end{align*}
Fourth, like $h$, $Ph$ also runs over a reduced residue system modulo $Qk$.
From all these observations, we find that
\begin{align}\label{eq3d1}
\phi(k,m)&=\frac{1}{Q}\sum_{\substack{0\leq h \leq Qk-1\\\gcd(h,Qk)=1}}W(h,k) \exp{\left(\frac{2\pi i}{Q k}\left((24P'd_{k}\ell_{k}m+jQ)H(h,Qk)-Qnh)\right)\right)}\nonumber\\
&=\frac{1}{Q}\sum_{\substack{0\leq h \leq Qk-1\\\gcd(h,Qk)=1}}W(Ph,k) \exp{\left(\frac{2\pi i}{Q k}\left((24P'd_{k}\ell_{k}m+jQ)P'H(h,Qk)-PQnh)\right)\right)}.
\end{align}
Substituting the values of $W(Ph,k)$ from \eqref{eq3a6} into \eqref{eq3d1}, we have the following
\begin{align}\label{eq_phi}
\lvert\phi(k,m)\rvert\leq Q^{-1}\lvert \mathcal{K}(\alpha,\beta;Qk)\rvert,
\end{align}
where we have the Kloosterman sum (see, \eqref{eq_KS}) on the right with $j=h$, $j'=- H(h,Qk)$, 
\begin{align*}
\alpha=
\begin{cases}       
-PQn, &  \text{if }k\text{ is odd},\\
3\ell k^2\left(1-\frac{\ell}{\ell_k^2}\right)+\frac{Qk}{8}\left(1-\frac{\ell}{\ell_k}\right)-PQn, &  \text{if }k\text{ is even},
\end{cases}
\end{align*}
and
\begin{align*}
\beta=24P'^{2}\left(d_k\ell_k(\delta_k-m)-\frac{k^2(\ell-1)}{8}\right)-jP'Q.
\end{align*}
To use Lemma \ref{lemma_Weil}, we estimate $\gcd\left(\alpha,\beta,Qk\right)$, with $\alpha$ and $\beta$ as defined above. For simplicity, let $g:=\gcd\left(\alpha,\beta,Qk\right)$. We consider two cases depending on the parity of $k$.\\
Case I: When $\alpha=-PQn=-48\ell n$. In this case, $g\mid48\ell n$\textcolor{red}{,} and therefore $g\ll_{\ell,n}1$.\\
Case II. When $\alpha=3\ell k^2\left(1-\frac{\ell}{\ell_k^2}\right)+\frac{Qk}{8}\left(1-\frac{\ell}{\ell_k}\right)-PQn$. In this case, $g\mid\alpha$ implies that
$$g\mid 16\alpha=48\ell k^2\left(1-\frac{\ell}{\ell_k^2}\right)+2Qk\left(1-\frac{\ell}{\ell_k}\right)-16PQn.$$
But, since $g\mid Qk$ \textcolor{red}{,}and $Qk\mid PQk=48\ell k$, $g$ also divides
$$48\ell k^2\left(1-\frac{\ell}{\ell_k^2}\right)+2Qk\left(1-\frac{\ell}{\ell_k}\right)=(48\ell k)\frac{k}{\ell_k}\left(\frac{\ell_k^2-\ell}{\ell_k}\right)+2Qk\left(1-\frac{\ell}{\ell_k}\right).$$
Therefore, $g\mid 16PQn=768\ell n$. Hence, $g\ll_{\ell,n}1$.\\
From both the cases above, we conclude that $\gcd\left(\alpha,\beta,Qk\right)\ll_{\ell,n}1$. Lemma \ref{lemma_Weil} and \eqref{eq_phi} yield
\begin{align*}
\phi(k,m)\ll_{\ell,n,\varepsilon}(Qk)^{\frac{1}{2}+\varepsilon}.
\end{align*}
Finally, from \eqref{eq_Phi}, we obtain 
\begin{align*}
\Phi(k,m,t)\ll_{\ell,n,\varepsilon}k^{\frac{1}{2}+\varepsilon}.
\end{align*}
This completes the proof of the lemma.
\end{proof}
In order to estimate $I_2(n;N)$, we need estimates for $a(m,k)$, $\Phi(k,m,t)$, and $\psi_{m,k}(\zeta)$, see \eqref{eq_new0}. With \eqref{eq_a_m,k_bound} and Lemma \ref{lemma_Phi}, it is only left to estimate $\psi_{m,k}(\zeta)$. For that, we note that since $\Re(\zeta)\leq1$\textcolor{red}{,} for $\lvert\zeta-\frac{1}{2}\rvert\leq\frac{1}{2}$, we have
\begin{align*}
\psi_{m,k}(\zeta)\ll\exp{\left(\frac{\pi d_{k}\ell_{k}(\delta_{k}-m)}{\ell\zeta}\right)},
\end{align*}
where the implicit constant is independent of $k$. For $\lvert\zeta-\frac{1}{2}\rvert\leq\frac{1}{2}$, we also have $\Re\left(\frac{1}{\zeta}\right)\geq1$. Therefore, 
\begin{align}\label{eq_bound_psi}
\psi_{m,k}(\zeta)\ll1~(\text{when }m\geq\delta_{k}) \text{ and } \psi_{m,k}(\zeta)\ll\exp{\left(\frac{-\pi m}{2\ell}\right)}~(\text{when }m\geq2\delta_{k}).
\end{align}
\par We are now ready to estimate $I_2(n;N)$. We rewrite \eqref{eq_new0} as follows
\begin{equation}\label{eqc}
I_{2}(n,N)=i\sum_{k=1}^{N}\sqrt{\frac{\ell_{k}}{\ell}}\sum_{m\geq\delta_{k}}a(m,k)\left(\frac{T_{1}}{k^2}+\frac{T_{2}}{k^2}+\frac{T_{3}}{k^2}\right).
\end{equation}
Using the estimation of $\Phi(k,m,t)$ and $\psi_{m,k}(\zeta)$ from Lemma \ref{lemma_Phi} and \eqref{eq_bound_psi}, respectively, we derive
\begin{align}\label{eq_new1}
\left|\frac{T_{1}}{k^2}\right|&\leq \sum_{t=N+1-k}^{N-1}\frac{1}{k^{2}}\int_{\zeta(t)}^{\zeta(t+1)}|\psi_{m,k}(\zeta)||\Phi(k,m,t)|d\zeta\nonumber\\
&\leq \frac{k^{1/2+\varepsilon}}{k^2}\int_{\zeta(N+1-k)}^{\zeta(N-1)}|\psi_{m,k}(\zeta)|d\zeta\nonumber\\
&\leq {\frac{k^{1/2+\varepsilon}}{k^2}}\exp{\left(\frac{-\pi m}{2\ell}\right)}\int_{\zeta(N+1-k)}^{\zeta(N-1)}d\zeta\nonumber\\
&\leq \frac{k^{1/2+\varepsilon}}{k^2}{\frac{k^2}{N^2}}\exp{\left(\frac{-\pi m}{2\ell}\right)}=\frac{k^{1/2+\varepsilon}}{N^2}\exp{\left(\frac{-\pi m}{2\ell}\right)},
\end{align}
where we use the fact that
\begin{align*}
|\zeta(N)-\zeta(N-k+1)|&\ll\frac{k^{2}}{N^{2}}.
\end{align*}
Similarly, one can find that
\begin{align}\label{eq_new2}
\left|\frac{T_{2}}{k^2}\right|\ll\frac{k^{1/2+\varepsilon}}{N^2}\exp{\left(\frac{-\pi m}{2\ell}\right)},
\end{align}
and 
\begin{align}\label{eq_new3}
\left|\frac{T_{3}}{k^2}\right|\ll\frac{k^{1/2+\varepsilon}}{k^2}\exp{\left(\frac{-\pi m}{2\ell}\right)}\int_{\zeta(-N)}^{\zeta(N)}d\zeta\ll\frac{k^{1/2+\varepsilon}}{Nk^2}\exp{\left(\frac{-\pi m}{2\ell}\right)},
\end{align}
where we use the following:
\begin{align*}
|\zeta(k-N-1)-\zeta(N)|&\ll\frac{k^2}{N^2},\\
|\zeta(N)-\zeta(-N)|&\ll\frac{1}{N}.
\end{align*}
Now, the use of \eqref{eq_new1}, \eqref{eq_new2}, and \eqref{eq_new3}, in \eqref{eqc} gives us
\begin{align*}
I_{2}(n;N)&\ll\sum_{k=1}^{N}\left(\frac{k^{1/2+\varepsilon}}{N^2}+\frac{k^{-3/2+\varepsilon}}{N}\right) \left(\sum_{ \delta_{k}\leq m\leq2\delta_k}|a(m,k)|+\sum_{m>2\delta_{k}}|a(m,k)|\right)\exp{\left(\frac{-\pi m}{2\ell}\right)}\\
&\ll\sum_{k=1}^{N}\left(\frac{k^{1/2+\varepsilon}}{N^2}+\frac{k^{-3/2+\varepsilon}}{N}\right)\left(1+\sum_{m>2\delta_{k}}|a(m,k)|\exp{\left(\frac{-\pi m}{2\ell}\right)}\right)\\
&\ll N^{-\frac{1}{2}+\varepsilon},
\end{align*}
where we use \eqref{eq_a_m,k_bound} noting that the implicit constant there doesn't depend on $k$.
Hence,
\begin{equation}
I_{2}(n;N)={O}\left(N^{\frac{-1}{2}+\varepsilon}\right).\nonumber
\end{equation}
Then, from \eqref{eq_new4}
\begin{align*}
\overline{A}_{\ell}(n)=I_{1}(n;N)+I_{2}(n;N)=I_{1}(n;N)+{O}\left(N^{\frac{-1}{2}+\varepsilon}\right),
\end{align*}
and therefore,
\begin{align*}
\overline{A}_{\ell}(n)=I_{1}(n;N)
\end{align*}
as $N\to\infty$.
\subsection{The main term $I_1(n;N)$}
Our attention now turns to the main contribution for $\overline{A}_{\ell}(n)$ by the main term associated with $0\leq m<\delta_{k}$. Since $\delta_k=0\textcolor{red}{,}$ for even values of $k$, we only consider the odd values of $k$ in the main term. We have
\begin{align}\label{eqd}
I_{1}(n;N)=&\sum_{\substack{k=1\\k \text{ odd}}}^{N}\frac{i}{k^2}\sqrt{\frac{\ell_{k}}{\ell}}\sum_{m=0}^{\lfloor\delta_{k}\rfloor}a(m,k){\sum_{\substack{h=0\\\gcd(h,k)=1}}^{k-1}}W(h,k) \times\exp\left(\frac{\pi i}{\ell k}(\ell_{k}mH(h,k)-2\ell nh)\right)\\ \nonumber
& \times \int_{\zeta_{1}(h,k)}^{\zeta_{2}(h,k)}\psi_{m,k}(\zeta)d\zeta.
\end{align}
Throughout, $K$ denotes the circle of radius $1/2$ and centered at $1/2$, while $K^{-}$ denotes this contour traversed in the clockwise direction. We use the following decomposition to each integral of \eqref{eqd}, as shown in Figure \ref{fig:placeholder}
\begin{align}\label{eq_new7}
\int_{\zeta_{1}(h,k)}^{\zeta_{2}(h,k)}=\int_{K^{-}}-\int_{0}^{\zeta_{1}(h,k)}-\int_{\zeta_{2}(h,k)}^{0}.
\end{align}

\begin{figure}
        \centering
        \includegraphics[width=0.3\linewidth]{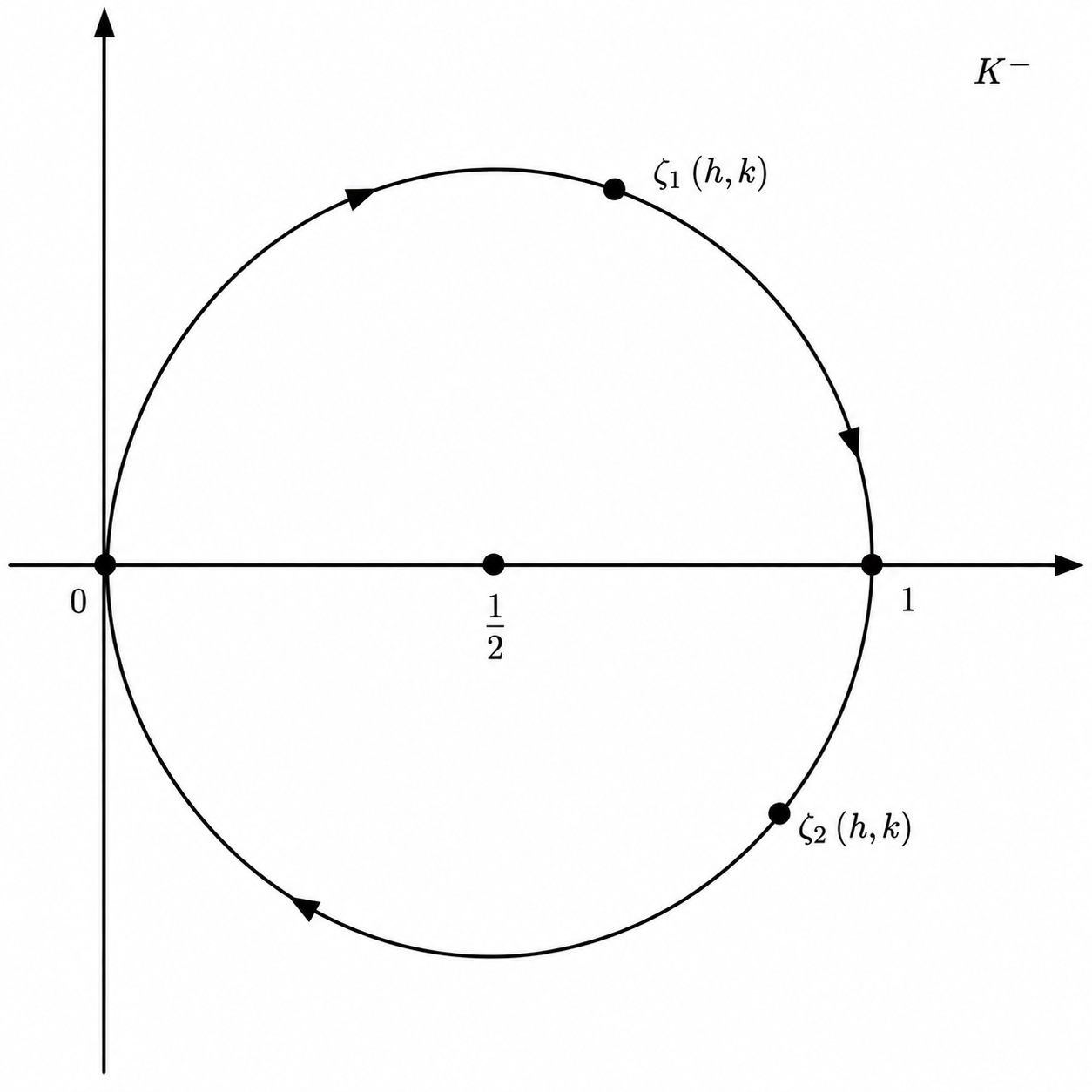}
    \caption{Decomposition of integral from $\zeta_{1}(h,k)$ to $\zeta_{2}(h,k)$}
    \label{fig:placeholder}
\end{figure}
We will show that the second and third integrals on the R.H.S. of \eqref{eq_new7} make negligible contributions to the integral over the arc from $\zeta_{1}(h,k)$ to $\zeta_{2}(h,k)$, and that only the first integral, i.e., the integral over $K^{-}$ matters. Arguing exactly as in \eqref{eqc}, we conclude that the contribution due to integrals $-\int_{0}^{{\zeta}_{1}(h,k)}-\int_{\zeta_{2}(h,k)}^{0}$ is bounded by 
\begin{align*}
I_{3}(n;N):=\sum_{\substack{k=1\\k \text{ odd}}}^{N}\frac{1}{k^2}\sum_{m=0}^{\lfloor\delta_{k}\rfloor}|T_{1}^{'}(m,k)|+|T_{2}^{'}(m,k)|+|T^{'}_{3}(m,k)|,
\end{align*}
where 
\begin{align*}
T_{1}^{'}(m,k)&=\sum_{t=N+1-k}^{N-1}\int_{\zeta(t)}^{\zeta(t+1)}\psi_{m,k}(\zeta)\Phi(k,m,t)d\zeta,\\ 
T_{2}^{'}(m,k)&=\sum_{t=-N}^{k-N-2}\int_{\zeta(t)}^{\zeta(t+1)}\psi_{m,k}(\zeta)\Phi(k,m,t)d\zeta,\\
T_{3}^{'}(m,k)&=\int_{\zeta(-N)}^{\zeta(N)}\psi_{m,k}(\zeta)\Phi(k,m,t)d\zeta,
\end{align*}
with $\psi_{m,k}(\zeta)$ and $\Phi(k,m,t)$ given by \eqref{eq_defn_psi} and \eqref{eq_defn_Phi}, respectively. Using Lemma \ref{lemma_Phi}, and the estimates
\begin{align*}
|\zeta(N)-\zeta(N-k+1)|\ll\frac{k}{N},~|\zeta(-N)-\zeta(k-N-1)|\ll\frac{k}{N},~|\zeta(N)-\zeta(-N)|\ll\frac{k}{N},
\end{align*}
along with $|\psi_{m,k}(\zeta)|\ll1$ (for $m<\delta_{k}$), we obtain
$$|T_{j}^{'}(m,k)|\ll k^{\frac{1}{2}+\varepsilon}\frac{k}{N}=\frac{k^{\frac{3}{2}+\varepsilon}}{N},$$ for all $j\in\{1,2,3\}.$
It follows that 
\begin{align*}
I_{3}(n,N)&\ll\sum_{\substack{k=1\\ k\,\text{odd}}}^{N}\frac{1}{k^2}\frac{k^{\frac{3}{2}+\varepsilon}}{N}\nonumber=\frac{1}{N}\sum_{k=1}^{N}k^{\frac{-1}{2}+\varepsilon}\ll N^{\frac{-1}{2}+\varepsilon}.\end{align*}
This proves our claim about the negligible contribution of  $-\int_{0}^{\zeta_{1}(h,k)}-\int_{\zeta_{2}(h,k)}^{0}$ to \eqref{eqd}. Therefore, we replace $\int_{\zeta_{1}(h,k)}^{\zeta_{2}(h,k)}$ with $\int_{K^{-}}$ in \eqref{eqd} to get
\begin{align}\label{eq_166}
\overline{A}_{\ell}(n)&=\sum_{\substack{k=1\\k \text{ odd}}}^{N}\frac{i}{k^{2}}\sqrt{\frac{\ell_{k}}{\ell}}\sum_{m=0}^{\lfloor\delta_{k}\rfloor}A_{k,m}(n) \int_{K^{-}}\exp{\left(\frac{\pi \ell_{k}(\delta_{k}-m)}{\ell \zeta}+\frac{2\pi n\zeta}{k^2}\right)}d\zeta+O\left(N^{\frac{-1}{2}+\varepsilon}\right),
\end{align} 
where $A_{k,m}(n)$ is defined by \eqref{eqe}. As $N\to\infty$, \eqref{eq_166} becomes
\begin{align}\label{eqfnew}
\overline{A}_{\ell}(n)&=\sum_{\substack{k=1\\k \text{ odd}}}^{\infty}\frac{i}{k^{2}}\sqrt{\frac{\ell_{k}}{\ell}}\sum_{m=0}^{\lfloor\delta_{k}\rfloor}A_{k,m}(n) \int_{K^{-}}\exp{\left(\frac{\pi \ell_{k}(\delta_{k}-m)}{\ell \zeta}+\frac{2\pi n\zeta}{k^2}\right)}d\zeta.
\end{align}  
To calculate this, consider a change of variable $\zeta=\frac{1}{w}$. Then limits of integration change from the circular path $K^{-}$ to a vertical trajectory ranging from $1-i\infty$ to $1+i\infty$. Hence, \eqref{eqfnew} can be transformed into \begin{equation}\label{eqg}
 \overline{A}_{\ell}(n)=\sum_{\substack{k=1\\ k \,\text{odd}}}^{\infty}\sum_{m=0}^{\lfloor\delta_{k}\rfloor}A_{k,m}(n)\frac{i}{k^2}\sqrt{\frac{\ell_{k}}{\ell}} \int_{1-i\infty}^{1+i\infty} \exp{\left(\frac{\pi w \ell_{k}(\delta_{k}-m)}{\ell}+\frac{2\pi n}{wk^2}\right)}\left(\frac{-1}{w^2}\right) dw.
\end{equation}
Next, let $s:=\frac{\pi \ell_{k}(\delta_{k}-m)}{\ell}$, and make a change of variable considering $u=sw$.
Then, \eqref{eqg} can be expressed as 
\begin{align*}
\overline{A}_{\ell}(n)=&\sum_{\substack{k=1\\k \text{ odd}}}^{\infty}\sum_{m=0}^{\lfloor\delta_{k}\rfloor}A_{k,m}(n)\frac{2\pi s}{k^2}\sqrt{\frac{\ell_{k}}{\ell}}\frac{1}{2\pi i} \int_{s-i\infty}^{s+i\infty}\frac{1}{u^2}\exp{\left(u+\frac{2\pi^2 \ell_{k}n(\delta_{k}-m)}{k^{2}\ell}\frac{1}{u}\right)}du\\
=&\sum_{\substack{k=1\\k \text{ odd}}}^{\infty}\sum_{m=0}^{\lfloor\delta_{k}\rfloor}\frac{2\pi A_{k,m}(n)}{k}\sqrt{\frac{\ell_{k}^{2}(\delta_{k}-m)}{2\ell^2n}} \frac{1}{2\pi i}\sqrt{\frac{2\pi^2 \ell_{k}n(\delta_{k}-m)}{k^{2}\ell}}\\
&\times\int_{s-i\infty}^{s+i\infty}u^{-1-1}\exp{\left(u+\frac{\sqrt{\frac{2\pi^2\ell_{k}n(\delta_{k}-m)}{k^2\ell}}}{u}\right)}du.
\end{align*}
Now, set
$$\frac{\zeta}{2}=\sqrt{\frac{2\pi^2\ell_{k}n(\delta_{k}-m)}{k^{2}\ell}}.$$
 It follows that
 \begin{align}
\overline{A}_{\ell}(n)&=\sum_{\substack{k=1\\k \text{ odd}}}^{\infty}\sum_{m=0}^{\lfloor\delta_{k}\rfloor}\frac{2\pi A_{k,m}(n)}{k}\sqrt{\frac{\ell_{k}^{2}(\delta_{k}-m)}{2\ell^{2}n}}\frac{1}{2\pi i}\left(\frac{\zeta}{2}\right)\int_{s-i\infty}^{s+i\infty}u^{-1-1}\exp{\left(u+\frac{\zeta^{2}}{4u}\right)}du\nonumber\\\label{eqh}
&=\sum_{\substack{k=1\\k \text{ odd}}}^{\infty}\sum_{m=0}^{\lfloor\delta_{k}\rfloor}\frac{2\pi A_{k,m}(n)}{k}\sqrt{\frac{\ell_{k}^{2}(\delta_{k}-m)}{2\ell^{2}n}}\times I_{1}\left(\frac{4\pi}{k}\sqrt{\frac{\ell_{k}n(\delta_{k}-m)}{2\ell}}\right).
\end{align}
Here, $I_{1}(s)$ is the modified Bessel function of the first kind, defined by \eqref{eq_BF_1}.
This proves Theorem \ref{main_theorem} and \eqref{eqh} gives the desired Rademacher-type formula for the function $\overline{A}_{\ell}(n)$.
\section{Proofs of Corollary \ref{cor_1} and Theorem \ref{Theorem2}}\label{Sec4}
In this section, we prove Corollary \ref{cor_1} and Theorem \ref{Theorem2}.
\begin{proof}[Proof of Corollary \ref{cor_1}] We write \eqref{eq1} as
\begin{align}\label{eq5.0}
\overline{A}_{\ell}(n)=M(n)+E(n),
\end{align}
where $M(n)$ is the main term corresponding to $k=1$ and $m=0$ (and hence $\ell_k=1$), and $E(n)$ is the error term due to the remaining terms. The main term $M(n)$ is explicitly 
\begin{align}\label{eq5.1}
M(n)=\frac{\pi}{2\ell}\sqrt{\frac{\ell-1}{n}}I_1\left(\pi\sqrt{n\left(1-\frac{1}{\ell}\right)}\right).
\end{align}
Using the asymptotic expansion \eqref{eq_I_asym} in \eqref{eq5.1}, we get
\begin{align}\label{eq5.2}
M(n)\sim\frac{1}{2\sqrt{2\ell}}\left(1-\frac{1}{\ell}\right)^{\frac{1}{4}}\left(\frac{1}{n}\right)^{\frac{3}{4}}e^{\pi\sqrt{n\left(1-\frac{1}{\ell}\right)}}
\end{align}
as $n\rightarrow\infty$. For the remaining terms in \eqref{eq1}, let $\beta=\frac{\ell_k}{2\ell}(\delta_k-m)$. We use \eqref{eq_I_asym}, once again, to find the following asymptotic expansion of $I_1$ in \eqref{eq1} for any $k\geq2$:
\begin{align}\label{eq5.3}
I_1\left(\frac{4\pi}{k}\sqrt{\beta n}\right)\sim\frac{1}{\sqrt{\frac{8\pi^2}{k}\sqrt{\beta n}}}\exp\left(\frac{4\pi}{k}\sqrt{\beta n}\right),
\end{align}
as $n\rightarrow\infty$.
Then using \eqref{eq5.3}, we get the following asymptotic expansion of the terms in \eqref{eq1} for any $k\geq2$:
\begin{align*}
\frac{2\pi}{k}A_{k,m}(n)\sqrt{\frac{\ell_k\beta}{\ell n}}I_1\left(\frac{4\pi}{k}\sqrt{\beta n}\right)\sim\frac{1}{\frac{2\pi\sqrt{2k}}{\beta^{1/4}}n^{3/4}}\exp\left(\frac{4\pi}{k}\sqrt{\beta n}\right),
\end{align*}
where the R.H.S. goes to 0 as $n\rightarrow\infty$. This along with \eqref{eq5.0} and \eqref{eq5.2} proves Corollary \ref{cor_1}.
\end{proof}
\begin{proof}[Proof of Theorem \ref{Theorem2}]
From \eqref{eq2.3}, for any $r\geq0$, we obtain
\begin{align*}
\frac{\overline{A}_{\ell}(n+r)}{\overline{A}_{\ell}(n)}\sim\left(\frac{n}{n+r}\right)^{\frac{3}{4}}\exp\left(\pi\sqrt{(n+r)\left(1-\frac{1}{\ell}\right)}-\pi\sqrt{n\left(1-\frac{1}{\ell}\right)}\right),
\end{align*}
which, on taking logarithm on both sides, becomes
\begin{align}\label{eq5.6}
\log\left(\frac{\overline{A}_{\ell}(n+r)}{\overline{A}_{\ell}(n)}\right)\sim\pi\sqrt{1-\frac{1}{\ell}}\left(\sqrt{n+r}-\sqrt{n}\right)-\frac{3}{4}\log\left(1+\frac{r}{n}\right).
\end{align}
Next, we use the series
\begin{align*}
\log(1+x)=\sum_{j=1}^{\infty}(-1)^{j+1}\frac{x^j}{j}
\end{align*}
for $x=\frac{r}{n}$, and the Taylor series of
$\sqrt{1+x}=\sum_{j=0}^{\infty}\binom{1/2}{j}x^j$
for $x=\frac{r}{n}$ to get
\begin{align*}
\sqrt{n+r}-\sqrt{n}=\sqrt{n}\left(\sqrt{1+\frac{r}{n}}-1\right)=\sum_{j=1}^{\infty}\binom{1/2}{j}\frac{r^j}{n^{j-1/2}}.
\end{align*}
Thus, \eqref{eq5.6} becomes
\begin{align}\label{eq5.7}
\log\left(\frac{\overline{A}_{\ell}(n+r)}{\overline{A}_{\ell}(n)}\right)\sim\pi\sqrt{1-\frac{1}{\ell}}\sum_{j=1}^{\infty}\binom{1/2}{j}\frac{r^j}{n^{j-1/2}}+\frac{3}{4}\sum_{j=1}^{\infty}\frac{(-1)^jr^j}{jn^j}.
\end{align}
The R.H.S. of \eqref{eq5.7} can be written as
\begin{align}\label{eq5.8}
\left(\frac{\pi}{2}\sqrt{\left(1-\frac{1}{\ell}\right)\frac{1}{n}}-\frac{3}{4n}\right)r&-\left(\frac{\pi}{8}\sqrt{\left(1-\frac{1}{\ell}\right)\frac{1}{n^3}}-\frac{3}{8n^2}\right)r^2\\
&+\sum_{j=3}^{\infty}\left(\pi\sqrt{1-\frac{1}{\ell}}\binom{1/2}{j}\frac{1}{n^{j-1/2}}+\frac{3}{4}\frac{(-1)^j}{jn^j}\right)r^j.\nonumber
\end{align}
Comparing \eqref{eq5.8} with \eqref{eq3.2}, we take
\begin{align*}
b(n)=\frac{\pi}{2}\sqrt{\left(1-\frac{1}{\ell}\right)\frac{1}{n}}-\frac{3}{4n},\\
\delta(n)=\sqrt{\frac{\pi}{8}\sqrt{\left(1-\frac{1}{\ell}\right)\frac{1}{n^3}}-\frac{3}{8n^2}},
\end{align*}
and
\begin{align*}
c_j(n)=\pi\sqrt{1-\frac{1}{\ell}}\binom{1/2}{j}\frac{1}{n^{j-1/2}}+\frac{3}{4}\frac{(-1)^j}{jn^j},
\end{align*}
for all $j\geq3$. One can easily check that
$\lim_{n\to\infty}\frac{c_j(n)}{(\delta(n))^j}=0$ for $3\leq j\leq d$,
and $\lim_{n\to\infty}\frac{c_j(n)}{(\delta(n))^d}=0$ for $j> d$. This implies that all the assumptions of Theorem \ref{theorem_GORZ} are being satisfied and therefore, the Jensen polynomials $J^{d,n}_{\overline{A}_{\ell}}$ can be written in terms of Hermite polynomials. Since, Hermite polynomials are hyperbolic for sufficiently large values of $n$, $J^{d,n}_{\overline{A}_{\ell}}$ are also hyperbolic for sufficiently large values of $n$. This proves Theorem \ref{Theorem2}.
\end{proof}

\section{Concluding Remarks}\label{Sec5}
Peng, Zhang, and Zhong \cite{Helen25} studied log-concavity and third-order Tur\'{a}n inequalities for the $\ell$-regular overpartition function $\overline{A}_{\ell}(n)$ for $2\leq \ell \leq 9$. The authors \cite{Helen25} used the asymptotic formula for the Fourier coefficients of general eta-quotients derived by Chern \cite{Chern2019}, which restricted them to values $2\leq \ell \leq 9$. However, in this article, we establish Rademacher-type formula (and hence an asymptotic formula) for $\overline{A}_{\ell}(n)$ for all squarefree odd integers $\ell$. Peng, Zhang, and Zhong provided the exact values, say $n_{\ell}$ and $\overline{n}_{\ell}$, such that $\overline{A}_{\ell}(n)$ satisfies log-concavity for all $n\geq n_{\ell}$ and it satisfies the third-order Tur\'{a}n inequalities for all $n\geq \overline{n}_{\ell}$, when $2\leq \ell \leq 9$, see, \cite[Theorem 1.3]{Helen25}. They also conjectured \cite[Conjecture 5.1]{Helen25} the values of $n_{\ell}$ and $\overline{n}_{\ell}$ for $10\leq \ell \leq 19$. In Theorem \ref{Theorem2}, we prove that for $d\geq1$ and a squarefree odd integer $\ell$, $\overline{A}_{\ell}(n)$ satisfies $d$-th order Tur\'{a}n inequality for all but finitely many values of $n$. In particular, we prove log-concavity ($d=2$) and third-order Tur\'{a}n inequality ($d=3$) for $\overline{A}_{\ell}(n)$, for all squarefree odd integers $\ell$. It would be interesting to provide the explicit bound on $n$ for which the higher order Tur\'{a}n inequalities for $\overline{A}_{\ell}(n)$ are satisfied. More precisely, the following question arises naturally: Can we find a positive integer $n_{\ell,d}$ such that $\overline{A}_{\ell}(n)$ satisfies the $d$-th order Tur\'{a}n inequality for all $n\geq n_{\ell,d}$?

\end{document}